\documentclass[12pt]{article}

\usepackage{amsmath,amssymb,amsthm}
\usepackage{bbold}
\usepackage{bm}

\usepackage{hyperref}

\makeatletter
\NewCommandCopy\@@pmod\pmod
\DeclareRobustCommand{\pmod}{\@ifstar\@pmods\@@pmod}
\def\@pmods#1{\mkern4mu({\operator@font mod}\mkern 6mu#1)}
\makeatother

\makeatletter
\newcommand\blfootnote[1]{%
  \begingroup
  \renewcommand{\@makefntext}[1]{\noindent\makebox[1.8em][r]#1}
  \renewcommand\thefootnote{}\footnote{#1}%
  \addtocounter{footnote}{-1}%
  \endgroup
}
\makeatother

\DeclareMathOperator{\meas}{meas}

\newtheorem{theorem}{Theorem}[section]
\newtheorem{lemma}{Lemma}[section]
\newtheorem{corollary}{Corollary}[section]

\begin{document}

\title{Prime Solutions of Diagonal Diophantine Systems}

\author{Alan Talmage\footnote{atalmage@semo.edu 
\newline Southeast Missouri State University, Cape Girardeau, Missouri, USA \newline 2020 Mathematics Subject Classification: 11P32, 11P55. \newline Keywords: Hardy-Littlewood method, diagonal systems, Waring-Goldbach problem, Diophantine systems.}}
%
%
%
%
%

\maketitle

\begin{abstract}

An asymptotic formula for the number of prime solutions of a general diagonal system of Diophantine equations is established, contingent on the existence of an appropriate mean value bound and on local solvability.  In conjunction with the Vinogradov Mean Value Theorem this yields an asymptotic formula for solutions of Vinogradov systems and in conjunction with Hooley's work on seven cubes this yields a conditional result for the Waring-Goldbach problem on seven cubes of primes, contingent on Hooley's form of the Riemann hypothesis.


\end{abstract}


\section{Introduction}

The proofs of Vinogradov's mean value theorem by Bourgain, Demeter, and Guth~\cite{bdg} and by Wooley~\cite{wooley_beyond_vmvt} have opened up new frontiers in finding integer solutions to systems of Diophantine equations.  In this work, we develop corresponding results regarding prime solutions of such systems, advancing the state of the Waring-Goldbach problem for systems of Diophantine equations to match that of Waring's problem in many cases.  In particular, following the progress of Brandes and Parsell~\cite{brandes_parsell}, we provide results for prime solutions of Diophantine systems that explicitly depend on the strength of available mean value bounds, permitting future progress on mean values to immediately update the results given here.  As an example of the progress achieved, we prove an asymptotic formula for the number of representations of a positive integer as a sum of seven cubes of primes conditional on a Riemann-type hypothesis of Hooley, matching Hooley's work~\cite{hooley} for seven cubes of integers.

Consider a general diagonal system of $t$ equations in $s$ prime variables $(p_1,$ $\ldots, p_s)$
\begin{equation}\label{the_system}
	\sum_{i=1}^s u_{ij}p_i^{k_j} = 0, \qquad (1 \leq j \leq t),
\end{equation}
where the $u_{ij}$ are nonzero integer coefficients and the $k_j$ are distinct positive integers.  Let 
\[
	k = \displaystyle\max_{1 \leq j \leq t} k_j,
\]
\[
	K = \displaystyle\sum_{j=1}^t k_j
\]
be the maximum degree and the total degree of the system, respectively.

Any theorem counting solutions to such a system necessarily depends on the strengths of the relevant mean value bounds.  Rather than restricting our study to the small number of cases where good mean value bounds are known, we will state a result conditional on the strength of the mean value bound.  Let $J_{\ell,\mathbf{k}}(P)$ be the number of integral solutions to the system
\[
	\sum_{i=1}^\ell x_i^{k_j} = \sum_{i=1}^\ell y_i^{k_j}, \qquad (1 \leq j \leq t)
\]
with $1 \leq x_i, y_i \leq P$.  Let $s_\varepsilon(\mathbf{k})$ be the least integer $\ell$ such that $J_{\ell, \mathbf{k}}(P) \ll P^{2\ell-K+\varepsilon}$.

The problem of determining when a system (\ref{the_system}) admits local solutions modulo each prime $p$ is still open.  Since we focus here on the global problem, we will assume that local solutions exist.

Let $S_0$ be the set of solutions of the system~(\ref{the_system}) and let
\[
	R(P) = \sum_{\substack{(p_1, \ldots, p_s) \in S_0 \\}} \prod_{i=1}^s \log p_i
\]
count them with logarithmic weighting. 

\begin{theorem}\label{main_theorem} If
\begin{itemize}

\item the system (\ref{the_system}) has a nontrivial real solution,

\item for every prime $p$, the system (\ref{the_system}) has a solution over the reduced residue classes modulo $p$, and

\item $s \geq 2s_\varepsilon(\mathbf{k})+1,$

then $R(P) \sim CP^{s-K}$ for some positive constant $C = C(\mathbf{u}_{ij}, \mathbf{k})$.
\end{itemize}
\end{theorem}

This theorem is stated conditionally on the best known value of $s_\varepsilon(\mathbf{k})$ for maximal generality.  In cases where such mean value bounds are known, this immediately provides an explicit theorem.  Notably, the proof of the Vinogradov mean value theorem by Bourgain, Demeter, and Guth~\cite{bdg} provides mean values for Vinogradov systems, which, in combination with Theorem~\ref{main_theorem}, gives the following result.

\begin{corollary}\label{vinogradov_systems_theorem} If (\ref{the_system}) is a Vinogradov system, i.e., $\mathbf{k} = (1, 2, 3, \ldots, k)$, then if
\begin{itemize}

\item the system (\ref{the_system}) has a nontrivial real solution,

\item for every prime $p$, the system (\ref{the_system}) has a solution over the reduced residue classes modulo $p$, and

\item $s \geq k^2+k+1$,

then $R(P) \sim CP^{s-k(k+1)/2}$ for some positive constant $C = C(\mathbf{u}_{ij},\mathbf{k})$.
\end{itemize}
\end{corollary}

This result can also be applied to the single equation of the standard Waring-Goldbach problem, given the existence of sufficiently good mean value bounds.  Let $H(3)$ be the smallest $s$ such that all sufficiently large odd numbers can be expressed as the sum of $s$ cubes of primes, in keeping with standard notation for the Waring-Goldbach problem.  Mean value bounds of sufficient strength to apply our theorem are not known unconditionally.  Conditional on a specific Riemann-type hypothesis, which we shall denote HW, Hooley~\cite{hooley} has a mean value bound of the form $s_\varepsilon(\{3\}) \le 3$.  (We make no attempt to elucidate the nature of this hypothesis here, since the key strength of Theorem~\ref{main_theorem} is that is allows us to treat the mean value bound as a black box.)  Combining this with Theorem~\ref{main_theorem} (and the observation that the definition of $H(3)$ automatically takes care of the local conditions) yields the following:

\begin{corollary}\label{seven_cubes} Assuming Hooley's hypothesis HW, $H(3) \le 7$.
\end{corollary}


We prove the main theorem by a circle method argument.  Section~\ref{sec:prelims} sets up the circle method and collects several needed lemmas from the literature.  Section~\ref{sec:exp_int} is devoted to the proof of an exponential integral bound of sufficient precision to treat the minor arcs.  Section~\ref{sec:minor_arcs} uses the assumed mean value bounds and Vaughan's identity to treat the minor arcs.  Section~\ref{sec:major_arcs} performs the major arc breakdown into a singular series and a singular integral.  Section~\ref{sec:asymp_formula} treats the singular series and integral to obtain an asymptotic formula.  Section~\ref{sec:main_theorem} concludes the argument and the proof of the main theorem.

\section{Preliminaries}\label{sec:prelims}

The letter $p$ will always be assumed to refer to a prime.  The exponential function $e(\alpha)$ is taken to mean $e^{2\pi i \alpha}$.  The letter $\varepsilon$ refers to a sufficiently small positive constant and the letter $C$ to a positive constant: these will not refer to consistent values from one place to another.  The functions $\Lambda$ and $\mu$ are those of von Mangoldt and M\"obius respectively.  Symbols in bold are tuples, with the corresponding symbol with a subscript denoting a component.  Operations performed on tuples are to be taken componentwise and multiplication is taken to mean the standard inner product.  Thus, for example,
\[
	e(\bm\alpha \mathbf{u}_i p^\mathbf{k}) = e\left(\sum_{j=1}^t \alpha_j u_{ij} p^{k_j}\right).
\]
We write $f(x) \ll g(x)$ for $f(x) = O(g(x))$, $f(x) \asymp g(x)$ if both $f(x) \ll g(x)$ and $g(x) \ll f(x)$ hold, and $f(x) \sim g(x)$ if $f(x)/g(x) \rightarrow 1$ as $x \rightarrow \infty$.

We assume without loss of generality that for each $j$, $\gcd(u_{1j}, \ldots, u_{sj})$ $= 1$.

Define the generating function
\[
	f_i(\bm\alpha; P) = f_i(\bm\alpha) = \sum_{p \leq P} (\log p)e(\bm\alpha \mathbf{u}_i p^\mathbf{k}).
\]
We will typically suppress the argument $P$ when there is no risk of ambiguity.  Much of our analysis will be blind to the coefficients $u_{ij}$, so we will also work with the coefficient-independent generating function
\[
	f(\bm\alpha) = \sum_{p \leq P} (\log p)e(\bm\alpha p^\mathbf{k}).
\]

Let $\mathcal{A} = (\mathbb{R}/\mathbb{Z})^t$.  Now
\begin{equation}\label{R_breakdown}
	\int_\mathcal{A} \prod_{i=1}^s f_i(\bm\alpha) d\bm\alpha
\end{equation}
\[
	= \int_\mathcal{A} \sum_{p_1, \ldots, p_s \leq P} \prod_{i=1}^s (\log p_i)e(\bm\alpha \mathbf{u}_i p^\mathbf{k}) d\bm\alpha
\]
\[
	= \sum_{\substack{(p_1, \ldots, p_s) \in S_0 \\}} \prod_{i=1}^s \log p_i = R(P)
\]
by orthogonality.  We divide $\mathcal{A}$ into major and minor arcs.  Let $Q = P^\delta$ for some sufficiently small $\delta>0$ to be fixed later.
 and for all $q < Q$, $1 \leq a_j \leq q$ for $1 \leq j \leq t$, $\gcd(a_1, \ldots, a_t, q) = 1$, let a typical major arc $\mathfrak{M}(\mathbf{a},q;Q)$ consist of all $\bm\alpha$ such that for each $j$ with $1 \leq j \leq t$,

\[
	\left|\alpha_j-\frac{a_j}{q}\right| \leq \frac{Q}{qP^{k_j}}.
\]
Let the major arcs $\mathfrak{M}(Q)$ be the union of all such $\mathfrak{M}(\bm\alpha,q;Q)$, and let the minor arcs $\mathfrak{m}(Q)$ be the remainder of $\mathcal{A}$.

We now collect several results from the literature and preliminary lemmas that follow directly from them.

\begin{lemma}\label{zero_density_estimates_lemma} Let $N(\alpha, T;\chi)$ denote the zeros $\rho=\beta+i\gamma$ of $L(s;\chi)$ with $\beta\ge\alpha$, $|\gamma|\le T$ .  Then there is a positive constant $c$ such that, uniformly for $\frac12\le \alpha\le 1$,
\[
\sum_{\chi\pmod* q} N(\alpha,T;\chi)\ll (qT)^{c(1-\alpha)}
\]
and
\[
\sum_{q\le Q}\,\,\sideset{}{^*}\sum_{\chi^*\,\mathrm{mod}\, q} N(\alpha,T;\chi^*)\ll (QT)^{c(1-\alpha)}
\]
where the implicit constant is absolute and $\sideset{}{^*}\sum$ indicates that the sum is restricted to primitive characters modulo $q$.
\end{lemma}

This follows from Theorem 1.1 of~\cite{fogels}.  (See also Theorem 6 of~\cite{gallegher}.)

\begin{lemma}\label{vmvt} (Vinogradov Mean Value Theorem) Let $F(\bm\alpha) = \displaystyle\sum_{n \leq N} e(\alpha_1 n + \alpha_2 n^2 + \ldots + \alpha_k n^k)$.  Then
\[
	\int_{[0,1]^k} |F(\bm\alpha)|^{2m} d\bm\alpha \ll N^{m+\varepsilon} + N^{2m-\frac{k(k+1)}{2} + \varepsilon}.
\]
Moreover, if $s > k(k+1)$, 
\[
	\int_{[0,1]^k} |F(\bm\alpha)|^s d\bm\alpha \ll N^{s-\frac{k(k+1)}{2}}.
\]
\end{lemma}

The first statement is Theorem 1.1 of~\cite{bdg}.  The second is formula (7) of the same paper.

Let
\[
	W(q, \mathbf{a}, \chi) = \sum_{r=1}^q e\left(\frac{\mathbf{a} r^\mathbf{k}}{q}\right)\chi(q)
\]
and
\[
	W_i(q, \mathbf{a}, \chi) = \sum_{r=1}^q e\left(\frac{\mathbf{a} \mathbf{u}_i r^\mathbf{k}}{q}\right)\chi(q)
\]

\begin{lemma}\label{cochrane_zheng_W_bound}
\[
	W_i(q, \mathbf a, \chi) \ll q^{1-\frac{1}{k+1} + \varepsilon}.
\]
\end{lemma}

This follows from Corollary 1.1 of~\cite{cochrane_zheng}.

\begin{lemma}\label{W_mvt} 

Suppose that $\chi$ is a Dirichlet character modulo $q$ and that $s \geq 2s_\varepsilon(\mathbf k)+1$.  Then
\[
	\sum_{a_{k_1}=1}^q\cdots\sum_{a_{k_t}=1}^q |W(q;\mathbf a,\chi)|^{s-1} \ll q^{s-1+\varepsilon}.
\]
\end{lemma}
\begin{proof}
By the orthogonality of the additive characters our sum is
\[
q^t \sideset{}{^*}\sum_{x_1,\ldots,x_K,y_1,\ldots y_K}\chi(x_1,\ldots x_K)\overline\chi(y_1,\ldots,y_K)
\]
where $\sideset{}{^*}\sum$ indicates that the sum is over $\mathbf x$ and $\mathbf y$ satisfying
\begin{align*}
x_1^{k_1}+\cdots +x_s^{k_1}&\equiv y_1^{k_1}+\cdots +y_s^{k_1}\pmod{q},\\
&\vdots\\
x_1^{k_t}+\cdots +x_s^{k_t}&\equiv y_1^{k_t}+\cdots +y_s^{k_t}\pmod q
\end{align*}
and $1\le x_j,y_j\le q$ with $(x_jy_j,q)=1$.  This sum is bounded by the number of solutions of
\begin{align*}
x_1^{k_1}+\cdots +x_s^{k_1} - & y_1^{k_1}-\cdots -y_s^{k_1} = h_1q,\\
&\vdots\\
x_1^{k_t}+\cdots +x_s^{k_t} - & y_1^{k_t}-\cdots -y_s^{k_t} = h_tq,
\end{align*}
with $1\le x_i,y_i\le q$ and $|h_j|\le Kq^{j-1}$.
For any given $\mathbf h$, by the definition of $s_\varepsilon(\mathbf k)$, the  number of such solutions is
\[
\ll q^{s-1-K+\varepsilon}.
\]
The number of $h_j$ for which there is a solution is 
\[
\ll q^{K-t}.
\]
The lemma follows.
\end{proof}

\section{An Exponential Integral}\label{sec:exp_int}



We will need the following exponential integral bound for the analysis of the major arcs.

\begin{theorem}\label{I_bound} Let $k \geq 2$, $\theta_1, \ldots, \theta_k, \tau$, and $\beta$ be real numbers such that $(k+1)/(k+2) \leq \beta \leq 1$, $\tau \neq 0$, and $\theta_k \neq 0$.  Let $\rho = \beta + 2\pi i \tau$ and let $X \geq 1$ be a real number.  Then
\[
	I(X; \bm\theta, \rho) = \int_0^X e(\theta_1 x + \cdots + \theta_k x^k) x^{\rho-1} dx
\]
satisfies
\[
	I(X; \bm\theta, \rho) \ll \frac{X^{\beta}}{(1 + X|\theta_1| + \cdots + X^k|\theta_k| + |\tau|)^{1/(1+k)}}.
\]

\end{theorem}

This theorem may be of independent interest, so it is stated in some generality.  The crux of the proof, however, is encapsulated in the following restricted version.

\begin{lemma}\label{restricted_I_bound} Let $k \geq 2$ and let $\theta_1, \ldots, \theta_k$, and $\tau$ be real numbers such that $\tau\theta_k \neq 0$.  Then
\[
	I( \bm\theta, \tau) = \int_1^2 e(\theta_1 x + \cdots + \theta_k x^k + \tau \log x) dx
\]
satisfies
\[
	I(\bm\theta, \tau) \ll (1 + |\theta_1| + \cdots + |\theta_k| + |\tau|)^{-1/(1+k)}.
\]

\end{lemma}

Define
\[
	Y_j = (|\theta_j| + \ldots + |\theta_k| + |\tau|)^{1/(k+1)}
\]
and let $Y_k = |\tau|^{1/(k+1)}$.  Let
\[
	h_1(x) = \theta_1 + 2\theta_2 x + \ldots + k\theta_k x^{k-1} + \frac{\tau}{x}.
\]
For a suitable sequence of constants $c_j$ with $c_0 = 1/2$, if $h_j$ is defined, define subsets of $[1,2]$
\[
	\mathcal{C}_j = \{x\in[1,2] : |h_j(x)| < 2c_{j-1}Y_j^j\}.
\]
If $\mathcal{C}_j \neq \emptyset$, choose $x_j\in\mathcal{C}_j$ and define
\[
	h_{j+1}(x) = \frac{h_j(x) - h_j(x_j)}{x-x_j}.
\]
Also define
\[
	\mathcal{D}_j = \{x\in[1,2] : |x-x_j| > Y_{j+1}^{-1}, |h_j(x)| < c_jY_{j+1}^j\}.
\]

\begin{lemma}\label{hj_structure}
If $h_j$ is defined, then
\begin{itemize}
\item $h_j(x)$ does not depend on $\theta_1, \ldots, \theta_{j-1}$, 
\item the function $h_j$ is linear each of in $\theta_j, \ldots, \theta_k, \tau$,
\item in the constant term of $h_j$, $\theta_j$ only appears as $j\theta_j$, and
\item the $\tau$-dependent term of $h_j$ is
\[
	\frac{\tau(-1)^{j+1}}{x x_1 \cdots x_{j-1} j}.
\]
\end{itemize}
\end{lemma}
\begin{proof}
The function $h_1$ immediately satisfies the lemma, so we may proceed by induction on $j$.  Assume the lemma holds for $h_j$ and that 
\[
	h_{j+1}(x) = \frac{h_j(x) - h_j(x_j)}{x-x_j}
\] 
as above.

Linearity in each variable follows immediately from the linearity of $h_j$.

The $\theta_j$-dependence of $h_{j+1}$ is $(j\theta_j - j\theta_j)/(x-x_j) = 0$.

To see that $\theta_j$ appears only as $j\theta_j$ in the constant term, note that the highest-order term involving $\theta_i$ will always be $i\theta_i x^{j-i}$.\

The $\tau$-dependent term of $h_{j+1}(x)$ is
\[
	\frac{\frac{(-1)^{j+1}\tau}{xx_1\cdots x_{j-1}} - \frac{(-1)^{j+1}\tau}{x_jx_1\cdots x_{j-1}}}{x-x_j}
\]
\[
	= \frac{(-1)^j \tau}{xx_1 \cdots x_j}.
\]
This completes the lemma.

\end{proof}

We prove our result by iteratively handling the integral on subsets of $[1,2]$ and showing that the ``bad'' set is small.  This iteration is encapsulated in the following lemma.

\begin{lemma}\label{Cj_iteration}
For $1 \leq j < k$, $\meas (\mathcal{C}_j) \leq Y_{j+1}^{-1} + \meas (\mathcal{C}_{j+1})$.
\end{lemma}
\begin{proof}
If $\mathcal{C}_j$ is empty the result is immediate.  So we may choose $x_j \in \mathcal{C}_j$ and define $h_{j+1}$ as above.  By Lemma~\ref{hj_structure}, there exists a constant $C$ such that for $x \in \mathcal{C}_j$, $h_j(x) \geq j|\theta_j| - CY_{j+1}^{k+1}$.  Also, by the definition of $\mathcal{C}_j$, $h_j(x) < 2c_{j-1}Y_j^j$.  By definition, $Y_j = (|\theta_j| + Y_{j+1}^{k+1})^{1/(k+1)}$, so $h_j(x) < 2c_{j-1}(|\theta_j| + Y_{j+1}^{k+1})^{j/(k+1)}$.

If $|\theta_j| > c_jY_{j+1}^{k+1}$ for a sufficiently large constant $c_j$, then $(c_j - C)Y_{j+1}^{k+1} < h_j(x) < 2c_{j-1}(C+1)^{1/(k+1)}Y_{j+1}^j$.  Choosing a sufficiently large $c_j$ now makes $\mathcal{C}_j$ empty, so we may assume that $|\theta_j| \leq c_jY_{j+1}^{k+1}$ for some sufficiently large constant $c_j$.

Thus we have $|h_j(x)| < (c_j + C)Y_{j+1}^{k+1}$.  But $Y_j = (|\theta_j| + Y_{j+1}^{k+1})^{1/(k+1)}$ and $h_j(x) < 2c_{j-1}Y_j^j$, so
\[
	|h_j(x)| < CY_{j+1}^j
\]
for some constant $C$.  The $x\in\mathcal{C}_j$ with $|x-x_j| \leq Y_{j+1}^{-1}$ contribute measure $\leq Y_{j+1}^{-1}$ to $\mathcal{C}_j$, and the rest is now $\mathcal{D}_j$.  Defining $h_{j+1}$ as above yields that $\mathcal{D}_j \subseteq \mathcal{C}_{j+1}$, so we may conclude
\[
	\meas (\mathcal{C}_j) \leq \frac{1}{Y_{j+1}} + \meas (\mathcal{C}_{j+1}).
\]

\end{proof}

\begin{lemma}\label{C1_bound}
For some constant $C$, $\meas (\mathcal{C}_1) \leq CY_1^{-1}$.
\end{lemma}
\begin{proof}
By Lemma~\ref{hj_structure} $h_k(x) = k\theta_k + \frac{(-1)^{k+1}\tau}{xx_1\cdots x_{k-1}}$, $\mathcal{C}_k = \{x : h_k(x) < 2c_{k-1}Y_k^k\}$, and $\mathcal{D}_k = \{x : |x-x_k| > |\tau|^{-1/(k+1)}, |h_k(x)| < c_k|\tau|^{k/(k+1)}\}$.  If $\mathcal{D}_k$ is empty, we are done.  Otherwise, let $y\in\mathcal{D}_k$.  Then
\[
	|h_k(x) - h_k(x_k)| = \left| \frac{\tau(x-x_k)}{xx_1\cdots x_{k-1}} \right| < 2c_k|\tau|^{-1/(k+1)},
\]
which yields
\[
	|x-x_k| < 2c_k xx_1\cdots x_{k-1}|\tau|^{-1/(k+1)} \leq 2^kc_k|\tau|^{-1/(k+1)}.
\]
Thus $\meas(\mathcal{D}_k) \leq 2^k c_k |\tau|^{-1/(k+1)}$.  Applying Lemma~\ref{Cj_iteration} now gives that
\[
	\meas (\mathcal{C}_1) \leq Y_1^{-1} + \ldots + Y_k^{-1} + 2^kc_k|\tau|^{-1/(k+1)} \leq 2^kc_kY_1^{-1}.
\]
\end{proof}

\noindent\textit{Proof of Lemma~\ref{restricted_I_bound}:}
Let $\mathcal{C}_0 = \{x\in[1,2] : |h_1(t)| > Y_1\}$.  Then $\mathcal{C}_0$ can be divided into $\ll 1$ intervals on which $h_1(x)$ does not change sign.  Thus by Lemma 4.3 of~\cite{titchmarsh}, 
\[
	\int_{\mathcal{C}_0} e(\theta_1 x + \cdots + \theta_k x^k + \tau \log x) dx \ll Y_1^{-1}.
\]
The remainder of $[1,2]$ is $\mathcal{C}_1$ and $\meas(\mathcal{C}_1) \leq CY_1^{-1}$, so 
\[
	I(\bm\theta, \tau) \ll Y_1^{-1}.
\]
\qed

\textit{Proof of Theorem~\ref{I_bound}}:
First, it suffices to prove the result in dyadic intervals, since if
\[
	\int_X^{2X} e(\theta_1x + \ldots + \theta_kx^k)x^{\rho-1} dx \ll \frac{X^\beta}{(1 + X|\theta_1| + \ldots + X^k |\theta_k| + |\tau|)^{1/(1+k)}},
\]
then
\[
	I(X; \bm\theta, \rho) = \int_0^{X} e(\theta_1x + \ldots + \theta_kx^k)x^{\rho-1} dx
\]
\[
	\ll \sum_{0 \leq \ell < \lceil \log_2 X \rceil} \frac{(X/2^\ell)^\beta}{(1 + (X/2^\ell)|\theta_1| + \ldots + (X/2^\ell)^k |\theta_k| + |\tau|)^{1/(1+k)}}.
\]
Let $X^j|\theta_j| = \max\{1, X|\theta_1|, \ldots, X^k|\theta_k|, |\tau|\}$ if $|\tau|$ is not the maximal term.  Then
\[
	I(X; \bm\theta, \rho) \ll \frac{(X/2^\ell)^\beta}{|\theta_j|(X/2^\ell)^{k/(k+1)}} \ll \frac{X^{\beta-j/(k+1)}}{|\theta_j|^{1/(k+1)}}.
\]
If $|\tau|$ is the maximal term, 
\[
	I(X; \bm\theta, \rho) \ll \sum_{0 \leq \ell \leq \lceil \log_2 X \rceil} \ll \frac{(X/2^\ell)^\beta}{|\tau|^{1/(k+1)}} \ll \frac{X^{\beta-j/(k+1)}}{|\tau|^{1/(k+1)}},
\]
as desired.

Next, it suffices to prove the result for $\beta = 1$, since, setting $\rho = 1 + 2\pi i \tau$,
\[
	I(X; \bm\theta, \rho) = \int_0^{X} e(\theta_1x + \ldots + \theta_kx^k)x^{\rho-1} dx
\]
\[
	\ll X^{\beta-1} \int_0^{X} e(\theta_1x + \ldots + \theta_kx^k)x^{2\pi i \tau} dx.
\]

Finally, it suffices to prove the result for $X=1$ by the variable change $x = Xu$.

This reduces the theorem to Lemma~\ref{restricted_I_bound}, completing the proof.

\qed

\section{The Minor Arcs}\label{sec:minor_arcs}



Our assumption of an existing mean value bound immediately provides the necessary Hua-type bound on the minor arcs.

\begin{lemma}\label{prime_mvt} 

If $s \geq 2s_\varepsilon(\mathbf{k})$, then
\[
	\int_{\mathcal{A}} |f(\bm\alpha)|^s d\bm\alpha \ll P^{2s-K+\varepsilon}.
\]
\end{lemma}
\begin{proof}
The integral 
\[
	\int_{\mathcal{A}} |f(\bm\alpha)|^{2\ell} d\bm\alpha
\]
counts the number of prime solutions of the system
\[
	\sum_{i=1}^\ell \alpha_i x_i^{k_j} = \sum_{i=1}^\ell \alpha_i y_i^{k_j}, \qquad 1 \leq j \leq t,
\]
with all $x_i, y_i \leq P$.  But the mean value $J_{s,\mathbf{k}}(P)$ counts the number of integer solutions to the same system, so the lemma follows from the definition of $s_\varepsilon(\mathbf{k})$.
\end{proof}

The remainder of this section will be dedicated to proving a pointwise minor arc bound of sufficient sensitivity for our purposes here.

We begin with a theorem on a general form of generating function.  (A version of this theorem first appeared in~\cite{talmage_dissertation}, but it is reproduced here for distribution to a wider audience and to correct some previous errors.)  Let $h(n) = e(\alpha_1 n + \ldots + \alpha_k n^k)$ and let
\[
	F_k(\bm\alpha) = \sum_{n \leq P} \Lambda(n) h(n).
\]

\begin{theorem}\label{F_bound} For some $\omega > 0$, if $\bm\alpha \in \mathfrak{m}(Q)$, then
\begin{equation}\label{F_definition}
	F_k(\bm\alpha) \ll P^{1-\omega}.
\end{equation}
\end{theorem}

By setting $\alpha_{k_j}$ to zero for each $k_j$ not in our system, we may immediately conclude from the above theorem the pointwise minor arc bound we desire.

\begin{corollary}\label{weyl_bound} For some $\omega > 0$, let $\bm\alpha \in \mathfrak{m}(Q)$.  Then
\[
	f(\bm\alpha) \ll P^{1-\omega}.
\]
\end{corollary}

We now proceed with the proof of Theorem~\ref{F_bound}.  Let $X = (\log P)^B$ for some value $B > 0$ to be chosen later.  Applying Vaughan's identity~\cite{vaughan_identity} to (\ref{F_definition}) yields

\begin{equation}\label{vaughans_identity_breakdown}
	F_k(\bm\alpha) = S_1 + S_2 + S_3 + S_4,
\end{equation}
where
\[
	S_1 = \sum_{n \leq X} \Lambda(n)h(n),
\]
\[
	S_2 = \sum_{n \leq P} \left( \sum_{\substack{ml=n \\ m \leq X}} \mu(m) \log l \right) h(n),
\]
\[
	S_3 = \sum_{n \leq P} \sum_{\substack{ml=n \\ m \leq X^2}} \left( \sum_{\substack{n_1,n_2 \\ n_1n_2=m \\ n_1 \leq X, n_2 \leq X}} \Lambda(n_1) \mu(n_2) \right) h(n),
\]
\[
	S_4 = \sum_{n \leq P} \left( \sum_{\substack{ml=n \\ m>X, l>X}} a(m)b(l) \right) h(n),
\]
with
\[
	a(m) = \sum_{\substack{l | m \\ l>X}} \Lambda(l),
\]
\[
	b(l) = \begin{cases} \mu(l), & l>X \\ 0, & l \leq X. \end{cases}
\]

We now bound each of the four sums $S_1, S_2, S_3$, and $S_4$ individually.

\begin{lemma}\label{S1_lemma}
\begin{equation}
	S_1 \ll X.
\end{equation}
\end{lemma}
\begin{proof}
Since $|h(n)| \ll 1$,
\[
	S_1 = \sum_{n \leq X} \Lambda(n)h(n) \ll \sum_{m \leq X} \Lambda(n) \ll X,
\]
where the last bound is a classical result of Chebyshev.
\end{proof}

\begin{lemma}\label{S3_lemma}
\[
	S_3 \ll P^{1 + \frac{Bk}{k^2+k+1} + 2B - \frac{\delta}{2k(k^2+k+1)}} (\log P).
\]
\end{lemma}
\begin{proof}

\begin{equation}\label{S3_def}
	S_3 = \sum_{n \leq P} \sum_{\substack{ml=n \\ m \leq X^2}} \left( \sum_{\substack{n_1,n_2 \\ n_1n_2=m \\ n_1 \leq X, n_2 \leq X}} \Lambda(n_1) \mu(n_2) \right) h(n).
\end{equation}
Let
\[
	c_3(m) = \sum_{\substack{n_1,n_2 \\ n_1n_2=m \\ n_1 \leq X, n_2 \leq X}} \Lambda(n_1) \mu(n_2)
\]
and note for future reference that
\[
	|c_3(m)| \leq \sum_{n_1|m}\Lambda(n_1) = \log m.
\]
Interchanging the order of summation in (\ref{S3_def}) yields
\[
	S_3 = \sum_{m \leq X^2} c_3(m) \sum_{l \leq P/m} h(ml)
\]
\begin{equation}\label{type_I_sum}
	= \sum_{m \leq X^2} c_3(m) \sum_{l \leq P/m} e(\alpha_1 ml + \ldots + \alpha_k m^k l^k).
\end{equation}

By Dirichlet's theorem on Diophantine approximation there exist $b_j$, $q_j$ for $1 \leq j \leq k$ such that $(b_j, q_j)=1$,
\begin{equation}\label{alpha_j_approximation}
	\left|\alpha_j m^j - \frac{b_j}{q_j}\right| \leq \frac{(P/m)^\delta}{q_j(P/m)^{k_j}},
\end{equation}
\[
	q_j \leq \frac{(P/m)^{j}}{(P/m)^\delta}.
\]
Assume for contradiction that $q_j \leq (P/m)^\delta$ for all $1 \leq j \leq k$ and rewrite (\ref{alpha_j_approximation}) as
\[
	\left|\alpha_j - \frac{b_j}{m^j q_j}\right| \leq \frac{(P/m)^\delta}{q_j P^{j}}.
\]
Let $b_j' = b_j/(m^j,b_j)$, $q_j' = m^j q_j/(m^j,b_j)$ for each $j$.  Then
\[
	\left|\alpha_j - \frac{b_j}{m^j q_j}\right| \leq \frac{(P/m)^\delta}{q_j' P^{j}},
\]
$(b_j',q_j')=1$, and $q_j' \leq (P/m)^\delta$ for each $j$.  Let $q = \text{lcm}(q_1', \ldots, q_t')$ and $a_j = b_j'q/q_j$.  Then  $(a_1, \ldots, a_t, q)=1$, $q \leq (\log(P/m))^A$, and
\[
	\left|\alpha_j - \frac{a_j}{q}\right| \leq \frac{(P/m)^\delta}{qP^j}.
\]
This implies that $\bm\alpha \in \mathfrak{M}(Q)$.  However, we have $\bm\alpha \in \mathfrak{m}(Q)$, which is the desired contradication, so we may assume that $q_{j_0} > (P/m)^\delta$ for at least one $j_0$ with $1 \leq j_0 \leq k$.
\

Let 
\begin{equation}\label{H_def}
	H(\bm\alpha, X) = \sum_{l \leq X} e(\alpha_1l + \ldots + \alpha_k l^k).
\end{equation}
We need a bound on $H(\bm\alpha, P/m)$, with $m \leq (\log P)^{2B}$.  Theorem 5.2 of~\cite{vaughanHLM} gives
\begin{multline}\nonumber
	H(\bm\alpha, P/m) \ll \Bigg(J_{l,k-1}(2P/m)\left(\frac{P}{m}\right)^{\frac{k(k-1)}{2}} \\ \times \prod_{j=1}^t\left(\frac{q_jm^k}{P^k} + \frac{m}{P} + \frac{1}{q_j}\right)\Bigg)^{1/2l} \log\left(\frac{2P}{m}\right)
\end{multline}
for any positive integer $l$.  Choose $l > k(k+1)$, so that by Lemma~\ref{vmvt}, $J_{l,k-1}(2P/m) \ll (P/m)^{2l-k(k-1)/2}$, which yields
\[
	H(\bm\alpha, P/m) \ll \left((P/m)^{2l}\prod_{j=1}^k\left(\frac{q_j'm^k}{P^k} + \frac{m}{P} + \frac{1}{q_j'}\right)\right)^{1/2l} \log(2P)
\]
\[
	\ll \frac{P}{m}\prod_{j=1}^k\left(\frac{q_j'm^k}{P^k} + \frac{m}{P} + \frac{1}{q_j'}\right)^{1/2l} \log(2P).
\]
Now $m/P \ll 1$, $1/q_j' \ll (P/m)^{-\delta/k} \ll 1$, and $q_j'm^j/P^{k_j} \ll (P/m)^{-\delta/k} \ll 1$, so for $j \neq j_0$, we have 
\[
	\frac{q_j'm^j}{P^{j}} + \frac{m}{P} + \frac{1}{q_j'} \ll 1.
\]
When $j = j_0$, we also have $1/q_j' \ll X^{2j}(P/m))^{-\delta/k}$, so
\[
	\frac{q_j'm^j}{P^{j}} + \frac{m}{P} + \frac{1}{q_j'} \ll \left((P/m)^{-\delta/k} + \frac{m}{P} + X^{2j_0}(P/m)\right)^{-\delta/k} \ll P^{2Bj_0}(P/m)^{-\delta/k}.
\]
Thus
\[
	H(\bm\alpha, P/m) \ll \frac{P}{m}P^{Bj_0/l}(P/m)^{-\delta/2l k} (\log 2P) \ll \frac{P}{m}P^{Bk/l}(P/m)^{-\delta/2kl} (\log 2P).
\]

Choosing $l = k^2+k+1 > k(k+1)$ yields
\begin{equation}\label{H_bound}
	H(\bm\alpha, P/m) \ll \frac{P}{m}P^{Bk/(k^2+k+1)}(P/m)^{\frac{-\delta}{2k(k^2+k+1)}} (\log 2P)
\end{equation}
Substituting this bound into (\ref{type_I_sum}), we obtain
\[
	S_3 \ll \sum_{m \leq X^2} \frac{P}{m}P^{Bk/l}(P/m)^{\frac{-\delta}{2k(k^2+k+1)}} (\log 2P)
\]
\[
	\ll P^{1 + \frac{Bk}{k^2+k+1} + 2B - \frac{\delta}{2k(k^2+k+1)}} (\log P).
\]

\end{proof}

\begin{lemma}\label{S2_lemma}
\[
	S_2 \ll P^{1 + \frac{Bk}{k^2+k+1} - \frac{\delta}{2k(k^2+k+1)}} (\log P)^3.
\]
\end{lemma}
\begin{proof}
\[
	S_2 = \sum_{n \leq P} \left( \sum_{\substack{ml=n \\ k \leq X}} \mu(m) \log l \right) h(n)
\]
\[
	= \sum_{m \leq X} \mu(m) \sum_{l < P/m} h(ml) \int_1^l \frac{dx}{x}
\]
\[
	= \sum_{m \leq X} \mu(m) \int_{P/X}^{P} \sum_{l < P/m} h(ml) \frac{dx}{x}
\]
\begin{equation}\label{S2_breakdown}
	= \int_{P/X}^{P} \left(\sum_{m \leq P/x} \mu(m) \sum_{l < P/m} h(ml) \right) \frac{dx}{x}.
\end{equation}
Now let $\bm\alpha' = (\alpha_1m, \ldots, \alpha_km^k)$ and by (\ref{H_bound}),
\[
	\sum_{l < P/m} h(ml) = H(\bm\alpha', P/m) \ll \frac{P}{m}P^{\frac{Bk}{k^2+k+1}}(P/m)^{\frac{-\delta}{2k(k^2+k+1)}} (\log 2P).
\]

Substituting this into (\ref{S2_breakdown}) yields
\[
	S_2 \ll \int_{P/X}^{P} \sum_{m \leq P/x} \mu(m) P^{\frac{Bk}{k^2+k+1}}\left(\frac{P}{m}\right)^{1-\frac{\delta}{2k(k^2+k+1)}} (\log 2P) \frac{dx}{x}
\]
\[
	\ll P^{\frac{Bk}{k^2+k+1}}P^{1-\frac{\delta}{2k(k^2+k+1)}} (\log 2P) (\log X) (\log P/m)
\]
\[
	\ll P^{1 + \frac{Bk}{k^2+k+1} - \frac{\delta}{2k(k^2+k+1)}} (\log P)^3.
\]
\end{proof}


\begin{lemma}\label{S4_lemma}
\[
	S_4 \ll P^{1-\frac{\min(\delta,B)}{4b^2}} (\log P).
\]
\end{lemma}
\begin{proof}
We begin by splitting $S_4$ into dyadic ranges.  Let $\mathcal{M} = \{X2^m : 0 \leq m, 2^m \leq P/X^2\}$.  Then
\begin{equation}\label{S4MDef}
	S_4 = \sum_{M \in \mathcal{M}} S_4(M),
\end{equation}
where 
\[
	S_4(M) = \sum_{M < m \leq 2M} \sum_{l \leq P/k} a(m)b(l)h(ml).
\]

Our goal is now to replace the sum over the range $l \leq P/m$ with one over the range $l \leq P/M$.  We begin by considering the integral
\[
	I(x) = \int_\mathbb{R} \frac{\sin(2\pi Ry)}{\pi y}e(-xy)dy,
\]
where $R>0$ is a constant.  Computing the integral via the residue theorem gives
\[
	I(x) = \begin{cases} 1, & |x|<R \\ 0, & |x|>R.\end{cases}
\]
Now for $x \neq R$, $y \geq 1$,
\begin{equation}\label{Tintegral}
	\int_{|y|>T} \frac{\sin(2\pi Ry)}{\pi y}e(-xy)dy = \int_{|y|>T} \frac{e\big((R-x)y\big)-e\big(-(R+x)y\big)}{2\pi iy} dy.
\end{equation}
Integrating the right hand side of (\ref{Tintegral}) by parts gives
\[
	\int_{|y|>T} \frac{\sin(2\pi Ry)}{\pi y}e(-xy)dy \ll \frac{1}{T|R-x|} + \frac{1}{T|R+x|} + \frac{1}{T^3} \ll \frac{1}{T\big|R-|x|\big|}.
\]
Thus we can rewrite $I(x)$ as an integral over $[-T,T]$ with an acceptable error term:
\[
	I(x) = \int_{-T}^T \frac{\sin(2\pi Ry)}{\pi y}e(-xy)dy + O\left(\frac{1}{T\big|R-|x|\big|}\right).
\]

We now take $R = \log(\lfloor P\rfloor + \frac{1}{2})$, $x = \log(ml)$, giving us
\[
	S_4(M) = \sum_{M < m \leq 2M} \sum_{l \leq P/M} a(m)b(l)h(ml)I(\log(ml))
\]
\[
	 = \int_{-T}^T \sum_{M < m \leq 2M} \sum_{l \leq P/M} \frac{a(m)b(l)}{(ml)^{2\pi iy}} h(kl) \frac{\sin(2\pi Ry)}{\pi y} dy + O\left(\frac{P^2\log P}{T}\right).
\]
Now 
\[
	\frac{\sin(2\pi Ry)}{\pi y} \ll \frac{1}{\pi y} \ll \frac{1}{|y|}
\]
and
\[
	\frac{\sin(2\pi Ry)}{\pi y} \ll \frac{2\pi Ry}{\pi y} \ll R,
\]
so
\[
	\frac{\sin(2\pi Ry)}{\pi y} \ll \min(R, 1/|y|).
\]
Take $T = P^3$, $a(m,y) = a(m)m^{-2\pi iy}$, $b(l,y) = b(l)l^{-2\pi iy}$, and let
\begin{equation}\label{S4MtDef}
	S_4(M, y) = \sum_{M < m \leq 2M} \sum_{l \leq P/m} a(m, y)b(l, y)h(ml).
\end{equation}
Then
\[
	S_4(M) \ll \sup_{|y|<T}|S_4(M,y)| \int_{-T}^T \frac{\sin(2\pi Ry)}{\pi y} dy
\]
\[
	\ll 1 + (\log P)\sup_{|y|<T}|S_4(M,y)|.
\]

We now consider $S_4(M,y)$.  Let $K = k(k+1)/2$ here and let $b>2K$.  By H\"older's inequality
\begin{multline}\label{procedureStart}
	S_4(M,y)^{2b} \ll \left(\sum_{M < m \leq 2M} |a(m,y)|^\frac{2b}{2b-1} \right)^{2b-1} \\ \times \sum_{M < m \leq 2M} \left| \sum_{l \leq P/M} b(l,y)h(ml) \right|^{2b}.
\end{multline}
Now $|a(k,y)| = |a(k)| \leq \log m \ll \log M \ll \log P$, so
\[
	S_4(M,y)^{2b} \ll \left( M(\log P)^\frac{2b}{2b-1}\right)^{2b-1} \sum_{M < m \leq 2M} \left| \sum_{l \leq P/M} b(l,y)h(ml) \right|^{2b}
\]
\begin{equation}\label{2b}
	\ll (\log P)^{2b} M^{2b-1} \sum_{M < m \leq 2M} \left| \sum_{l \leq P/M} b(l,y)h(ml) \right|^{2b}.
\end{equation}
Expanding the $2b$-th power in (\ref{2b}) yields
\[
	\left| \sum_{l \leq P/M} b(l,y)h(ml) \right|^{2b}
\]
\begin{equation}\label{introsl}
	= \sum_{\substack{\mathbf l \\ l_j \leq P/M}} \left( \prod_{i=1}^b b(l_i,y) \prod_{i=b+1}^{2b} \overline{b(l_i,y)} \right) e\left(\bm\alpha m \mathbf{s}(\mathbf l)\right)
\end{equation}
where $\mathbf{s}(\mathbf{l})$ is defined by
\[
	s_j(\mathbf l) = l_1^j + \ldots + l_b^j - l_{b+1}^j - \ldots - l_{2b}^j.
\]
Collecting terms in (\ref{introsl}) by values of $s_j$ yields
\begin{equation}\label{bfSumBound}
	\left| \sum_{l \leq P/M} b(l,y)h(ml) \right|^{2b} = \sum_{\substack{\mathbf{v} \\ |v_j| \leq bP^j}} R_1(\mathbf{v}) e(\bm\alpha m \mathbf{v})
\end{equation}
where
\[
	R_1(\mathbf{v}) = \sum_{\substack{\mathbf{l} \\ l_j \leq P/M \\ \mathbf{s}(\mathbf{l}) = \mathbf{v}}} \prod_{i=1}^b b(l_i,y) \prod_{i=b+1}^{2b} \overline{b(l_i,y)} \ll J_{b,3}(P/M) \ll (P/M)^{2b-K}
\]
by Lemma~\ref{vmvt}.
Substituting (\ref{bfSumBound}) into (\ref{2b}) yields
\[
	S_4(M,y)^{2b} \ll (\log P)^{2b}M^{2b-1} \sum_{\substack{\mathbf{v} \\ |v_j| \leq bP^jM^{-j}}} R_1(\mathbf{v}) \sum_{M < k \leq 2M} e(\bm\alpha m \mathbf{v})
\]
\begin{equation}\label{procedureEnd}
	\ll (\log P)^{2b}M^{K-1} P^{2b-K} \sum_{\substack{\mathbf{v} \\ |v_j| \leq bP^jM^{-j}}} \sum_{M < m \leq 2M} e(\alpha_1mv_1 + \ldots + \alpha_km^kv_k).
\end{equation}
We now repeat the procedure followed from (\ref{procedureStart}) to (\ref{procedureEnd}).  By H\"older's inequality
\begin{multline}
	S_4(M,y)|^{4b^2} \ll \left( (\log P)^{2b} M^{K-1} P^{2b-K} \right)^{2b} \left( \sum_{\substack{\mathbf{v} \\ |v_j| \leq bP^jM^{-j}}} 1^\frac{2b}{2b-1} \right)^{2b-1} \\ \times \sum_{\substack{\mathbf{v} \\ |v_j| \leq bP^jM^{-j}}} \left| \sum_{M < m \leq 2M} e(\alpha_1mv_1 + \ldots + \alpha_km^kv_k) \right|^{2b}
\end{multline}
\begin{multline}\nonumber
	\ll (\log P)^{4b}M^{2b(K-1)}P^{4b^2-2bK} \left( b^3P^{K}M^{-K} \right)^{2b-1} \\ \times \sum_{\substack{\mathbf{v} \\ |v_j| \leq bP^jM^{-j}}} \left| \sum_{M < m \leq 2M} e(\alpha_1mv_1 + \ldots + \alpha_km^kv_k) \right|^{2b}
\end{multline}
\begin{multline}\label{orNot2b}
	\ll (\log P)^{4b^2} M^{K-2b} P^{4b^2-K} \\ \times \sum_{\substack{\mathbf{v} \\ |v_j| \leq bP^jM^{-j}}}  \left| \sum_{M < m \leq 2M} e(\alpha_1mv_1 + \ldots + \alpha_km^kv_k) \right|^{2b}.
\end{multline}
We expand the $2b$-th power in (\ref{orNot2b}) and collect like terms. This yields
\[
	\left| \sum_{M < m \leq 2M} e(\alpha_1mv_1 + \ldots + \alpha_km^kv_k) \right|^{2b}
\]
\[
	= \sum_{\substack{\mathbf{m} \\ M < m_j \leq 2M}} e(\alpha_1s_1(\mathbf{m})v_1 + \ldots + \alpha_ks_k(\mathbf{m})v_k)
\]
\begin{equation}\label{eSumBound}
	= \sum_{\substack{\mathbf{u} \\ |u_j| \leq b2^jM^j}} R_2(\mathbf{u}) e(\bm\alpha \mathbf{u} \mathbf{v})
\end{equation}
where
\[
	R_2(\mathbf{u}) = \sum_{\substack{\mathbf{m} \\ M < m_j \leq 2M \\ \mathbf{s}(\mathbf{k}) = \mathbf{u}}} 1 \ll J_{b,k}(2M) \ll M^{2b-K}
\]
by Lemma~\ref{vmvt}.
Substituting (\ref{eSumBound}) into (\ref{orNot2b}), we obtain
\[
	S_4(M,y)^{4b^2} \ll (\log P)^{4b^2} P^{4b^2-K} \sum_{\substack{\mathbf{u} \\ |u_j| \leq b2^jM^j}} \left| \sum_{\substack{\mathbf{v} \\ |v_j| \leq bP^jM^{-j}}} e(\bm\alpha \mathbf{u} \mathbf{v}) \right|.
\]
Summing over each of the $v_j$ gives
\[
	S_4(M,y)^{4b^2} \ll (\log P)^{4b^2} P^{4b^2-K} \sum_{\substack{\mathbf{u} \\ |u_j| \leq b2^jM^j}} \prod_{j=1}^k \min\left(\frac{P^j}{M^j},\frac{1}{\|\alpha_ju_j\|}\right).
\]
Applying Lemma 2.2 of~\cite{vaughanHLM} yields
\begin{equation}\label{S4MtBound}
	S_4(M,y)^{4b^2} \ll (\log P)^{4b^2+k} P^{4b^2} \prod_{j=1}^k \left(\frac{1}{q_j} + \frac{1}{M^j} + \frac{M^j}{P^j} + \frac{q_j}{P^j}\right).
\end{equation}
Combining (\ref{S4MtBound}) with (\ref{S4MDef}) and (\ref{S4MtDef}), we obtain
\[
	S_4 \ll P(\log P) \prod_{j=1}^k \left(\frac{1}{q_j} + \frac{1}{X^j} + \frac{q_j}{P^j}\right)^{1/(4b^2)}.
\]
Recalling that $q_j > P^{\delta/t}$ for some $j$ and $X = P^B$, this is
\begin{equation}\label{S4bound}
	S_4 \ll P^{1-\min(\delta,B)/(4b^2)} (\log P)
\end{equation}
for $b > k(k+1)$.
\end{proof}


\noindent\textit{Proof of Theorem~\ref{F_bound}}: Using the Vaughan's identity breakdown of (\ref{vaughans_identity_breakdown}) and the estimates for the $S_i$ found in Lemmas~\ref{S1_lemma},~\ref{S3_lemma},~\ref{S2_lemma}, and~\ref{S4_lemma}, we have
\[
	F_3(\bm\alpha) = S_1 + S_2 + S_3 + S_4
\]
\begin{align*}
	\ll & P^B + P^{1 + \frac{Bk}{k^2+k+1} + 2B - \frac{\delta}{2k(k^2+k+1)}} (\log P) \\ + & P^{1 + \frac{Bt}{k^2+k+1} - \frac{\delta}{2k(k^2+k+1)}} (\log P)^3 + P^{1-\min(\delta,B)/(4b^2)} (\log P)
\end{align*}
for any $b > k(k+1)$.
We choose $1-\omega > B > 4b^2\omega$ and $\delta > 2k(k^2+k+1)\left(\frac{Bk}{k^2+k+1} + 2B + \omega\right)$ for a given $\omega > 0$ (noting that choosing $\omega, B$ sufficiently small allows us to have $\delta$ as small as desired), yielding
\[
	F_k(\bm\alpha) \ll P^{1-\omega}
\]
uniformly over $\bm\alpha \in \mathfrak{m}(Q)$.
\qed

We can now conclude the necessary bound on the minor arcs.

\begin{theorem}\label{minor_arc_bound} For $s \geq 2s_\varepsilon(\mathbf{k}) + 1$ and for some constant $\varepsilon > 0$,
\[
	\int_{\mathfrak{m}(Q)} |f(\bm\alpha)|^s d\bm\alpha \ll P^{s-K-\varepsilon}.
\]
\end{theorem}
\begin{proof}
This follows directly from Lemma~\ref{prime_mvt} and Corollary~\ref{weyl_bound}.
\end{proof}

\section{The Major Arc Breakdown}\label{sec:major_arcs}

To treat the major arcs, it is convenient to treat the central and outer portions of the major arcs separately.  Let $R = (\log P)^A$ for some positive constant $A$ and define the small major arcs 

\begin{equation}\label{inner_major_arc_def}
	\mathfrak{B}(\mathbf a, q) = \{\bm\alpha: q(1 + |\bm\theta|P^{\mathbf k})\leq R\},
\end{equation}
and the outer major arcs
\begin{equation}\label{outer_major_arc_def}
	\mathfrak N(\mathbf a, q) = \{\bm\alpha: R\leq q(1 + |\bm\theta|P^{\mathbf k})\leq Q\}.
\end{equation}
Let $\mathfrak{B}(Q)$ be the union of all such $\mathfrak{B}(\mathbf a, q)$, and let $\mathfrak{N}(Q)$ be the union of all such $\mathfrak{N}(\mathbf a, q)$.

Let
\begin{equation}\label{fi*_def}
	f_i^*(\bm\alpha) = \frac{1}{\phi(q)} W_i(q, \mathbf{a}) \int_{0}^P e(\bm\theta \mathbf{u}_i x^\mathbf{k}) dx
\end{equation}
be our desired approximation for $f_i(\bm\alpha)$ on the major arcs.  We first establish the approximation on the central major arcs $\mathfrak{B}(Q)$.

\begin{lemma}\label{small_major_arc_approximation} Let $\bm\alpha = \frac{\mathbf{a}}{q} + \bm\theta \in \mathfrak{B}(Q)$.  Then for some positive constant $C$ depending only on $A$,
\[
	|f_i(\bm\alpha) - f_i^*(\bm\alpha)| \ll P \exp(-C(\log P)^{1/2}).
\]
\end{lemma}
\begin{proof}
\[
	|f_i(\bm\alpha) - f_i^*(\bm\alpha)|
\]
\[
	= \left|\sum_{p \leq P} (\log P) e\left(\frac{\mathbf{a} \mathbf{u}_i r^\mathbf{k}}{q}\right)e(\bm\theta \mathbf{u}_i p^\mathbf{k}) - \frac{1}{\phi(q)} W_i(q, \mathbf{a}) \int_0^P e(\bm\theta \mathbf{u}_i x^\mathbf{k}) dx\right|
\]
\[
	= \left| \sum_{\substack{r=1 \\ (r,q)=1}}^q e\left(\frac{\mathbf{a} \mathbf{u}_i r^\mathbf{k}}{q}\right) \left( \sum_{\substack{p \leq P \\ p \equiv r \pmod* q}} (\log p) e(\bm\theta \mathbf{u}_i p^\mathbf{k}) - \frac{1}{\phi(q)} \int_0^P e(\bm\theta \mathbf{u}_i x^\mathbf{k}) dx \right) \right|
\]
\[
	= \left| \sum_{\substack{r=1 \\ (r,q)=1}}^q e\left(\frac{\mathbf{a} \mathbf{u}_i r^\mathbf{k}}{q}\right)  \sum_{m \leq P} \left((\log m)\mathbb{1}_{\mathcal{P}_{r}} - \frac{1}{\phi(q)} \right) e(\bm\theta \mathbf{u}_i m^\mathbf{k}) \right| + O(|\bm\theta|P^\mathbf{k})
\]
where $\mathbb{1}_{\mathcal{P}_{r}}$ is the indicator function of the primes $\equiv r \pmod* q$, i.e. $\mathcal{P}_r = \{p : p \equiv r\pmod q\}$.

Now by the Siegel-Walfisz theorem,
\[
	\sum_{m \leq P} (\log m)\mathbb{1}_{\mathcal{P}_{r}} - \frac{1}{\phi(q)} \ll P\exp(-C(\log P)^{1/2})
\]
for some constant $C$ depending only on $A$.  Thus.
\[
	|f_i(\bm\alpha) - f_i^*(\bm\alpha)|
\]
\[
	\ll \phi(q)P\exp(-C(\log P)^{1/2}) + O(|\bm\theta|P^\mathbf{k})
\]
\[
	\ll P\exp(-C(\log P)^{1/2}).
\]
\end{proof}

\begin{lemma}\label{major_arc_approx} For some positive constant $C$,
\[
	\int_{\mathfrak{B}(Q)} \left| \prod_{i=1}^s f_i(\bm\alpha) - \prod_{i=1}^s f_i^*(\bm\alpha) \right| d\bm\alpha \ll P^{s-K}\exp(-C(\log P)^{1/2}).
\]
\end{lemma}
\begin{proof}
\[
	\int_{\mathfrak{B}(Q)} \left| \prod_{i=1}^s f_i(\bm\alpha) - \prod_{i=1}^s f_i^*(\bm\alpha) \right| d\bm\alpha
\]
\[
	\sum_{q < Q} \mathop{\sum_{a_1=1}^q \ldots \sum_{a_t=1}^q}_{(a_1, \ldots, a_t, q)=1} \int_{\mathfrak{B}(q, \mathbf{a}, Q)} \left| \prod_{i=1}^s f_i(\bm\alpha) - \prod_{i=1}^s f_i^*(\bm\alpha) \right| d\bm\alpha
\]
\[
	\ll \sum_{q < Q} \mathop{\sum_{a_1=1}^q \ldots \sum_{a_t=1}^q}_{(a_1, \ldots, a_t, q)=1} \int_{-Q/P^{k_1}}^{Q/P^{k_1}} \ldots \int_{-Q/P^{k_t}}^{Q/P^{k_t}} P^s \exp(-C(\log P)^{1/2}) d\bm\theta
\]
\[
	\ll Q^{t+1} P^{s-K} \exp(-C(\log P)^{1/2})
\]
\[
	\ll P^{s-K} \exp(-C(\log P)^{1/2}).
\]
\end{proof}

We now consider the approximation on the outer major arcs $\mathfrak{N}(Q)$.

\begin{lemma}\label{outer_major_arc_approx}
\[
	\int_{\mathfrak{N}(Q)} \left| \prod_{i=1}^s f_i(\bm\alpha) - \prod_{i=1}^s f_i^*(\bm\alpha) \right| \ll P^{s-K-\varepsilon}.
\]
\end{lemma}
\begin{proof}

\[
	f_i(\bm\alpha) = \sum_{p \leq P} (\log p) e\left(\mathbf{u}_i \left(\frac{\mathbf{a}}{q} + \bm\theta\right) p^{\mathbf{k}}\right)
\]
\[
	= \sum_{\substack{r=1 \\ (r,q)=1}} e\left(\frac{\mathbf{u}_i \mathbf{a} r^\mathbf{k}}{q}\right) \sum_{\substack{p \leq P \\ p \equiv r \pmod* q}} (\log p) e(\mathbf{u}_i \bm\theta p^\mathbf{k}) + O(q^\varepsilon).
\]
Rewriting the sum over $p$ in terms of characters modulo $q$, this is
\begin{multline}\label{f_in_psi}
	f_i(\bm\alpha) = \sum_{\substack{r=1 \\ (r,q)=1}} e\left(\frac{\mathbf{u}_i \mathbf{a} r^\mathbf{k}}{q}\right) \frac{1}{\phi(q)} \sum_{\chi\,\mathrm{mod}\,q} \overline{\chi(r)} \sum_{p \leq P} \chi(p) \Lambda(p) e(\mathbf{u}_i \bm\theta p^\mathbf{k})\\ + O(P^{1/2}).
\end{multline}

Recall the definitions
\[
	W_i(q, \mathbf{a}) = \sum_{r=1}^q e\left(\frac{\mathbf{u}_i \mathbf{a} r^\mathbf{k}}{q}\right),
\]
\[
	W_i(q, \mathbf{a}, \chi) = \sum_{r=1}^q e\left(\frac{\mathbf{u}_i \mathbf{a} r^\mathbf{k}}{q}\right) \chi(r)
\]
and define
\[
	\psi(x, \chi) = \sum_{n \leq x} \chi(n) \Lambda(n),
\]
and
\[
	\psi(x, \chi, \bm\theta) = \sum_{n \leq P} \chi(n) \Lambda(n) e(\mathbf{u}_i \bm\theta n^\mathbf{k}).
\]
Thus
\[
	f_i(\bm\alpha) = \frac{1}{\phi(q)} \sum_{\chi \,\mathrm{mod}\, q} W_i(q, \mathbf{a}, \overline{\chi}) \psi(x, \chi, \bm\theta) + O(P^{1/2}).
\]
Applying Abel summation to $\psi(x, \chi, \bm\theta)$ yields
\begin{equation}\label{psi_post_abel}
	\psi(x, \chi, \bm\theta) = e(\mathbf{u}_i \bm\theta P^\mathbf{k}) \psi(P, \chi) - \int_0^P \psi(x, \chi) 2\pi i \mathbf{u}_i \bm\theta \mathbf{k} x^{\mathbf{k}-1} e(\mathbf{u}_i \bm\theta x^\mathbf{k}) dx.
\end{equation}

Let $c$ and $d$ be absolute constants with
\begin{equation}\label{c_d_def}
	0 < 3\delta < d < d + \frac{\delta}{k+1} < c < \frac{1}{2}. 
\end{equation}

Applying Theorem 12.12  of~\cite{montgomery_vaughan} in conjuction with formula (12.13) of the same source yields that
\[
	\psi(x, \chi) = E_\chi x - \sum_{\rho\in\mathcal{R}(Q,\chi)} \frac{x^{\rho}}{\rho} + O(P^{1-d}(\log P)^2),
\]
where
\[
\mathcal R(Q,\chi)=\{\rho:1-\delta\le \beta\le 1, |\gamma|\le T, L(\rho;\chi)=0\},
\]

\[
	E_\chi = \begin{cases} 0, & \chi \neq \chi_0, \\ 1, & \chi = \chi_0,\end{cases} 
\]
and $\chi_0$ is the principal character mod $q$.  
Inserting this bound into (\ref{psi_post_abel}) and integrating by parts yields
\begin{multline}
	\psi(x, \chi, \bm\theta) = E_\chi \int_0^P e(\mathbf{u}_i \bm\theta x^\mathbf{k}) dx - \sum_{\rho\in\mathcal{R}(Q,\chi)} \int_0^P x^{\rho-1} e(\mathbf{u}_i \bm\theta x^\mathbf{k}) dx \\ + O((1+P^\mathbf{k}|\bm\theta|)P^{1-d}(\log P)^2).
\end{multline}

Inserting this into (\ref{f_in_psi}), we obtain
\begin{align}\label{fi_breakdown}
	f_i(\bm\alpha) & = \frac{1}{\phi(q)} W_i(q, \mathbf{a}) \int_0^P e(\mathbf{u}_i \bm\theta x^\mathbf{k}) dx \nonumber \\ & - \frac{1}{\phi(q)} \sum_{\chi \,\mathrm{mod}\, q} W_i(q, \mathbf{a}, \overline{\chi}) \sum_{\rho\in\mathcal{R}(Q,\chi)} \int_0^P x^{\rho-1} e(\mathbf{u}_i \bm\theta x^\mathbf{k}) dx \\ & + O(W_i(q, \mathbf{a}, \overline{\chi})(1+P^\mathbf{k}|\bm\theta|)P^{1-d}(\log P)^2).\nonumber
\end{align}

The first term above is $f_i^*(\bm\alpha)$.  Define $I_\rho(\bm\theta) = I(P; \rho, \bm\theta)$, recalling the latter definition from Theorem~\ref{I_bound}, and let
\[
	g_i(\bm\alpha) = \sum_{\chi \,\mathrm{mod}\, q} \frac{W_i(q, \mathbf a, \overline{\chi})}{\phi(q)} \sum_{\rho \in \mathcal{R}(Q,\chi)} I_\rho(\bm\theta)
\]
and
\[
	0 \ne h_i(\bm\alpha) \ll W_i(q, \mathbf a, \overline\chi)(1 + P^{\mathbf k}|\bm\theta|)P^{1-d}(\log P)^2
\]
so that we have
\[
	f_i(\bm\alpha) = f_i^*(\bm\alpha) - g_i(\bm\alpha) + h_i(\bm\alpha).
\]

We now wish to bound the contribution of $g$ and $h$ on the outer major arcs $\mathfrak{N}(Q)$.  We have
\begin{equation}\label{f_f*_g_h}
	\prod_{i=1}^s f_i(\bm\alpha) \ll \prod_{i=1}^s |f_i^*(\bm\alpha)| + \prod_{i=1}^s |g_i(\bm\alpha)| + \prod_{i=1}^s |h_i(\bm\alpha)|.
\end{equation}
Now by Lemma~\ref{cochrane_zheng_W_bound} and (\ref{outer_major_arc_def}),
\[
	\prod_{i=1}^s h_i(\bm\alpha) \ll \left( q^{1 - \frac{1}{k+1} + \varepsilon} \left(\frac{Q}{q}\right) P^{1-d} (\log P)^2 \right)^s \ll P^{s-\varepsilon}, 
\]
so 
\begin{equation}\label{h_bound}
	\int_{\mathfrak{N}(Q)} \prod_{i=1}^s h(\bm\alpha) \ll P^{s-K-\varepsilon}.
\end{equation}
It remains only to bound the contribution from $g_i(\bm\alpha)$ on $\mathfrak{N}(Q)$.  We have 
\[
	\prod_{i=1}^s g_i(\bm\alpha) \ll |g_i(\bm\alpha)|^s
\]
for some $i$ with $1 \leq i \leq s$.  By H{\"o}lder's inequality,
\[
	|g_i(\bm\alpha)|^s \ll G_1 G_2^{s-1}
\]
where
\[
	G_1 = \sum_{\chi \,\mathrm{mod}\, q} \frac{|W_i(q; \mathbf a, \overline\chi)|^s}{\phi(q)^s} \sum_{\rho \in \mathcal{R}(Q,\chi)} |I_\rho(\bm\theta)|
\]
and
\[
	G_2 = \sum_{\chi \,\mathrm{mod}\, q} \sum_{\rho \in \mathcal{R}(Q,\chi)} |I_\rho(\bm\theta)|.
\]

By Theorem~\ref{I_bound},
\[
	I_\rho(\bm\theta) \ll \frac{P^{\beta}}{(\xi + |\gamma|)^\nu},
\]
where $\xi = 1 + P^{\mathbf k}|\bm\theta|$ and $\nu = \frac{1}{1+k}$.

Each term of $G_2$ is
\[
	|I_\rho(\bm\theta)| \ll \frac{P^{\beta}}{(\xi + |\gamma|)^\nu}
\]
\[
= \left(
P^{1-\delta}+\int_{1-\delta}^{\beta} P^{\eta}(\log P)d\eta
\right)\left(
\frac1{(\xi+Q)^{\nu}} + \int_{|\gamma|}^Q \frac{\nu d\tau}{(\xi+\tau)^{1+\nu}}
\right)
\]

\[
	= \frac{P^{1-\delta}}{(\xi+Q)^{\nu}}
\]
\[
	+ \int_{|\gamma|}^Q \frac{P^{1-\delta}\nu d\tau}{(\xi+\tau)^{1+\nu}}
\]
\[
	+ \int_{1-\delta}^{\beta} \frac{P^{\eta}}{(\xi+Q)^{\nu}}(\log P)d\eta
\]
\[
	+ \int_{1-\delta}^{\beta} \int_{|\gamma|}^Q \frac{P^{\eta}(\log P)\nu d\tau}{(\xi+\tau)^{1+\nu}} d\eta.
\]
By Lemma~\ref{zero_density_estimates_lemma}, for some positive constant $c$, the contributions from these four terms are bounded by
\[
\frac{P^{1-\delta}}{(\xi+Q)^{\nu}}\sum_{\chi\pmod* q} N(1-\delta,Q;\chi)\ll \frac{P^{1-\delta}(qQ)^{c\delta}}{(\xi+Q)^{\nu}},
\]
\[
\int_{0}^Q \sum_{\chi\pmod* q} \frac{N(1-\delta,\tau;\chi)P^{1-\delta}\nu d\tau}{(\xi+\tau)^{1+\nu}}\ll \int_{0}^Q \frac{P^{1-\delta}(q+q\tau))^{c\delta}\nu d\tau}{(\xi+\tau)^{1+\nu}},
\]
\[
	\int_{1-\delta}^{1} \sum_{\chi\pmod* q} \frac{N(\eta,Q;\chi)P^{\eta}}{(\xi+Q)^{\nu}}(\log P)d\eta \ll \int_{1-\delta}^{1} \frac{P^{\eta}(qQ)^{c(1-\eta)}}{(\xi+Q)^{\nu}}(\log P)d\eta,
\]
and
\begin{multline*}
\int_{1-\delta}^{1} \int_{0}^Q \sum_{\chi\pmod* q} N(\eta,\tau;\chi)\frac{P^{\eta}(\log P)\nu }{(\xi+\tau)^{1+\nu}} d\tau d\eta\\
\ll \int_{1-\delta}^{1} \int_{0}^Q \frac{P^{\eta}(\log P)(q+q\tau)^{c(1-\eta)}\nu }{(\xi+\tau)^{1+\nu}} d\tau d\eta
\end{multline*}
respectively.
\par
We have $q\le Q=P^\delta$, 
so $(q+q\tau)^{c(1-\eta)}\ll P^{2c\delta(1-\eta)}$, which is small compared with $1-\eta$, so on integrating with respect to $\eta$ we gain back a $\log$ factor.
Likewise $c(1-\eta)$ is small compared with $\nu$.  Thus the first and third terms contribute $\ll P\xi^{-\nu}$ and the second and fourth
\[
\ll \int_{0}^T \frac{\nu P}{(\xi+\tau)^{1+\nu}}d\tau \ll P\xi^{-\nu}.
\]
Thus it follows that $G_2 \ll P\xi^{-\nu}$.
\par
We now bound the contribution of $g$.  We begin with $G_1$:

\[
G_1 \ll  \sum_{r|q}\sideset{}{^*}\sum_{\chi^*\,\mathrm{mod}\,r} \frac{|W(q;\mathbf a,\chi_0\chi^*)|^{s}}{\phi(q)^{s}} \sum_{\rho\in\mathcal R(Q,\chi^*)} \frac{P^{\beta}}{(\xi+|\gamma|)^{\nu}}.
\]
Divide the prime factors of $q/r$ based on whether they divide $r$: define $g$ and $h$ by $q=rgh$, $p|g\Rightarrow p|r$ and $p|h\Rightarrow p\nmid r$.  Then
\begin{equation*}
\sum_{\substack{\mathbf a\,\mathrm{mod}\,q\\(\mathbf a,q)=1}} \frac{|W(q;\mathbf a,\chi_0\chi^*)|^{s}}{\phi(q)^{s}} 
\end{equation*}
\begin{equation*}
= \sum_{\substack{\mathbf a\,\mathrm{mod}\,rg\\(\mathbf a,rg)=1}} \frac{|W(rg;\mathbf a,\chi^*)|^{s}}{\phi(rg)^{s}} \sum_{\substack{\mathbf b\,\mathrm{mod}\,h\\(\mathbf b,h)=1}} \frac{|W(h;\mathbf b,\chi_0)|^{s}}{\phi(h)^{s}}.
\end{equation*}
By Lemmas~\ref{cochrane_zheng_W_bound} and~\ref{W_mvt} this is
\[
\ll (rg)^{\varepsilon-\frac{1}{k+1}} h^{\varepsilon-1-\frac{1}{k+1}}.
\]
We also recall the condition (\ref{outer_major_arc_def}) on $\mathfrak N(q,\mathbf a)$, namely that
\[
R\le rgh\xi\le Q.
\]
Rearranging the order of summation and integration we have
\begin{align*}
&\sum_{q\le Q} \sum_{\substack{
\mathbf a\,\mathrm{mod}\,q\\
(\mathbf a,q)=1
}} \int_{\mathfrak N(q,\mathbf a)} |g(\bm\theta)|^{s} d\bm\alpha\ll  \\
&\int_{\xi\le Q} \sum_{\frac{R}{\xi}\le rgh\le \frac{Q}{\xi}} (rg)^{\varepsilon-\frac{1}{k+1}} h^{\varepsilon-1-\frac{1}{k+1}} \sideset{}{^*}\sum_{\chi^*\,\mathrm{mod}\,r} \sum_{\rho\in\mathcal R(Q,\chi^*)} \frac{P^{\beta+s-1}d\bm\theta}{\xi^{(s-1)\nu}(\xi+|\gamma|)^{\nu}}.
\end{align*}
Extracting an $(rgh\xi)^{\delta}$ and summing over $h$ yields that this is
\[
\ll R^{-\delta} \int_{\xi\le Q} \sum_{rg\le Q} (rg)^{\varepsilon+\delta-\frac{1}{k+1}} \sideset{}{^*}\sum_{\chi^*\,\mathrm{mod}\,r} \sum_{\rho\in\mathcal R(Q,\chi^*)} \frac{P^{\beta+s-1}d\bm\theta}{\xi^{(s-1)\nu-\delta}(\xi+|\gamma|)^{\nu}}.
\]

By definition, all the prime factors of $g$ are prime factors of $r$.  Thus 
\[
\sum_{g} g^{\varepsilon+\delta-\frac{1}{k+1}} =\prod_{p|r}(1-p^{\varepsilon+\delta-\frac{1}{k+1}})\ll r^{\varepsilon}.
\]
So 
\[
	\sum_{q\le Q} \sum_{\substack{
\mathbf a\,\mathrm{mod}\,q\\
(\mathbf a,q)=1
}} \int_{\mathfrak N(q,\mathbf a)} |g(\bm\theta)|^{s} d\bm\alpha
\]
\[
 \ll R^{-\delta} \int_{\xi \leq Q} Q^{1+2\varepsilon+\delta-\nu} \sideset{}{^*}\sum_{\chi^*\,\mathrm{mod}\,r} \sum_{\rho\in\mathcal R(Q,\chi^*)} \frac{P^{\beta+s-1}d\bm\theta}{\xi^{(s-1)\nu-\delta}(\xi+|\gamma|)^{\nu}}
\]
\[
	\ll R^{-\delta} \int_{\xi \leq Q} \frac{Q^{1+2\varepsilon+\delta-\nu}P^{s-1}}{\xi^{(s-1)\nu-\delta}} \sideset{}{^*}\sum_{\chi^*\,\mathrm{mod}\,r} \sum_{\rho\in\mathcal R(Q,\chi^*)} \frac{P^{\beta}d\bm\theta}{(\xi+|\gamma|)^{\nu}}
\]
By using the second part of Lemma~\ref{zero_density_estimates_lemma}, this double sum can treated in the same way as $G_2$, yielding
\[
	\ll R^{-\delta} \int_{\xi \leq Q} Q^{1+2\varepsilon+\delta-\nu} P^s \int_{\xi \leq Q} \frac{1}{\xi^{\nu s - \delta}}
\]
\[
	\ll R^{-\delta} \int_{\xi \leq Q} Q^{1+2\varepsilon+2\delta-\nu + t - \nu s} P^{s-K}.
\]
Now $\varepsilon, \delta$ are small and $s-1 \geq 2k$, so $(s+1)\nu \geq 2$, so $1 + \varepsilon + 2\delta - s\nu \leq -\varepsilon$ and we can conclude that
\begin{equation}\label{g_bound}
	\int_{\mathfrak{N}(Q)} |g(\bm\alpha)|^{s} d\bm\alpha \ll P^{s-K-\varepsilon}.
\end{equation}

Combining (\ref{g_bound}) and (\ref{h_bound}) with (\ref{f_f*_g_h}) now completes the lemma.

\end{proof}

Define the singular series
\[
	\mathfrak{S}(Q) = \sum_{q < Q} \mathop{\sum_{a_1=1}^q \cdots \sum_{a_t=1}^q}_{(a_1, \ldots, a_t, q)=1} \frac{1}{\phi(q)^s} \prod_{i=1}^s W_i(q, \mathbf{a})
\]
and the singular integral
\[
	J(Q) = \int_{-Q/P^{k_1}}^{Q/P^{k_1}} \cdots \int_{-Q/P^{k_t}}^{Q/P^{k_t}} \prod_{i=1}^s \int_0^P e(\bm\theta \mathbf{u}_i x^\mathbf{k})dx d\bm\theta.
\]

\begin{theorem}\label{major_arc_breakdown} For some positive constant $C$,
\[
	R(P) = C\mathfrak{S}(Q) J(Q) + O(P^{s-K-\varepsilon})
\]

\end{theorem}
\begin{proof}
Recall from (\ref{R_breakdown}) and the definitions of $\mathfrak{B}(Q)$, $\mathfrak{N}(Q)$, and $\mathfrak{m}(Q)$ that
\[
	R(P) = \int_{\mathcal{A}}\prod_{i=1}^s f_i(\bm\alpha)d\bm\alpha
\]
\[
	= \int_{\mathfrak{B}(Q)}\prod_{i=1}^s f_i(\bm\alpha)d\bm\alpha + \int_{\mathfrak{N}(Q)}\prod_{i=1}^s f_i(\bm\alpha)d\bm\alpha + \int_{\mathfrak{m}(Q)}\prod_{i=1}^s f_i(\bm\alpha)d\bm\alpha
\]
Lemma~\ref{minor_arc_bound} handles the integral over the minor arcs, while Lemmas~\ref{major_arc_approx} and~\ref{outer_major_arc_approx} handle the major arc approximations.  This leaves us with
\[
	R(P) = \int_{\mathfrak{M}(Q)} \prod_{i=1}^s f_i^*(\bm\alpha) d\bm\alpha + O(P^{s-K-\varepsilon}).
\]
Recalling the definition of $f_i^*$ from (\ref{fi*_def}) and integrating over the major arcs yields the theorem.
\end{proof}

\section{The Asymptotic Formula}\label{sec:asymp_formula}

The singular integral is identical to that arising in Brandes and Parsell's work on the corresponding problem on the integers, so their bound can be used without modification.

\begin{lemma}\label{singular_integral_bound} For some positive constant $C$,
\[
	J(Q) \sim CP^{s-K}.
\]
\end{lemma}

This is Lemma 4.2 of~\cite{brandes_parsell} and the following discussion.

We now turn to the singular series.  Let

\[
	A(q) = \mathop{\sum_{a_1=1}^q \cdots \sum_{a_t=1}^q}_{(a_1, \ldots, a_t, q)=1} \frac{1}{\phi(q)^s} \prod_{i=1}^s W_i(q, \mathbf{a})
\]
so that
\[
	\mathfrak{S}(Q) = \sum_{q < Q} A(q).
\]

\begin{lemma} $A(q)$ is multiplicative.
\end{lemma}
\begin{proof}
Let $1 \leq q_1 < Q$, $1 \leq q_2 < Q$ such that $(q_1, q_2)=1$.  Then
\[
	A(q_1q_2) = \mathop{\sum_{a_1=1}^{q_1q_2} \cdots \sum_{a_t=1}^{q_1q_2}}_{(a_1, \ldots, a_t, q_1q_2)=1} \frac{1}{\phi(q_1q_2)^s} \prod_{i=1}^s W_i(q_1q_2, \mathbf{a}).
\]
We can now write each $a_j$ as $b_jq_2 + c_jq_1$ with $1 \leq b_j \leq q_1$, $(b_1, \ldots, b_t, q_1)=1$, $1 \leq c_j \leq q_2$, $(c_1, \ldots, c_t, q_2)=1$.  Thus
\[
	A(q_1q_2) = \frac{1}{\phi(q_1q_2)^s} \mathop{\sum_{b_1=1}^{q_1} \cdots \sum_{b_t=1}^{q_1}}_{(b_1, \ldots, b_t, q_1)=1}  \mathop{\sum_{c_1=1}^{q_2} \cdots \sum_{c_t=1}^{q_2}}_{(c_1, \ldots, c_t, q_2)=1} \prod_{i=1}^s \sum_{\substack{r=1 \\ (r,q_1q_2)=1}}^{q_1q_2} e\left(\frac{(\mathbf{b}q_2 + \mathbf{c}q_1)\mathbf{u}_i r^\mathbf{k}}{q_1q_2}\right)
\]
\begin{multline}\nonumber
	= \frac{1}{\phi(q_1q_2)^s} \mathop{\sum_{b_1=1}^{q_1} \cdots \sum_{b_t=1}^{q_1}}_{(b_1, \ldots, b_t, q_1)=1}  \mathop{\sum_{c_1=1}^{q_2}  \cdots \sum_{c_t=1}^{q_2}}_{(c_1, \ldots, c_t, q_2)=1} \\ \prod_{i=1}^s \sum_{\substack{r=1 \\ (r,q_1q_2)=1}}^{q_1q_2} e\left(\frac{(\mathbf{b}q_2)^\mathbf{k}\mathbf{u}_i r^\mathbf{k}}{q_1}\right) e\left(\frac{(\mathbf{c}q_1)^\mathbf{k}\mathbf{u}_i r^\mathbf{k}}{q_2}\right).
\end{multline}
Now $\mathbf{b}q_2$ and $\mathbf{c}q_1$ run through the same residue classes modulo $q_1$ and $q_2$ as $\mathbf{b}$ and $\mathbf{c}$, so
\begin{multline}\nonumber
	A(q_1q_2) = \frac{1}{\phi(q_1q_2)^s} \mathop{\sum_{b_1=1}^{q_1} \cdots \sum_{b_t=1}^{q_1}}_{(b_1, \ldots, b_t, q_1)=1}  \mathop{\sum_{c_1=1}^{q_2} \cdots \sum_{c_t=1}^{q_2}}_{(c_1, \ldots, c_t, q_2)=1}  \\ \prod_{i=1}^s \sum_{\substack{r=1 \\ (r,q_1q_2)=1}}^{q_1q_2} e\left(\frac{\mathbf{b}^\mathbf{k}\mathbf{u}_i r^\mathbf{k}}{q_1}\right) e\left(\frac{\mathbf{c}^\mathbf{k}\mathbf{u}_i r^\mathbf{k}}{q_2}\right).
\end{multline}
We can now express $r$ as $r_2q_1 + r_1q_2$ with $1 \leq r_1 \leq q_1$, $(r_1,q_1)=1$, $1 \leq r_2 < q_2$, $(r_2,q_2)=1$, which yields
\begin{multline}\nonumber
	A(q_1q_2) = \frac{1}{\phi(q_1q_2)^s} \mathop{\sum_{b_1=1}^{q_1} \cdots \sum_{b_t=1}^{q_1}}_{(b_1, \ldots, b_t, q_1)=1}  \mathop{\sum_{c_1=1}^{q_2} \cdots \sum_{c_t=1}^{q_2}}_{(c_1, \ldots, c_t, q_2)=1}  \\ \prod_{i=1}^s   \sum_{\substack{r_1=1 \\ (r_1,q_1)=1}}^{q_1} \sum_{\substack{r_2=1 \\ (r_2,q_2)=1}}^{q_2} e\left(\frac{\mathbf{b}^\mathbf{k}\mathbf{u}_i (r_1q_2)^\mathbf{k}}{q_1}\right) e\left(\frac{\mathbf{c}^\mathbf{k}\mathbf{u}_i (r_2q_1)^\mathbf{k}}{q_2}\right).
\end{multline}
As above, $\mathbf{b}q_2$ and $\mathbf{c}q_1$ can be replaced with $\mathbf{b}$ and $\mathbf{c}$, so we have
\begin{multline}
	A(q_1q_2) = \frac{1}{\phi(q_1q_2)^s} \mathop{\sum_{b_1=1}^{q_1} \cdots \sum_{b_t=1}^{q_1}}_{(b_1, \ldots, b_t, q_1)=1} \mathop{\sum_{c_1=1}^{q_2} \cdots \sum_{c_t=1}^{q_2}}_{(c_1, \ldots, c_t, q_2)=1}  \\ \prod_{i=1}^s \sum_{\substack{r_1=1 \\ (r_1,q_1)=1}}^{q_1} \sum_{\substack{r_2=1 \\ (r_2,q_2)=1}}^{q_2} e\left(\frac{\mathbf{b}^\mathbf{k}\mathbf{u}_i r_1^\mathbf{k}}{q_1}\right) e\left(\frac{\mathbf{c}^\mathbf{k}\mathbf{u}_i r_2^\mathbf{k}}{q_2}\right)
\end{multline}
\[
	= \frac{1}{\phi(q_1)^s}\frac{1}{\phi(q_2)^s} \mathop{\sum_{b_1=1}^{q_1} \cdots \sum_{b_t=1}^{q_1}}_{(b_1, \ldots, b_t, q_1)=1}  \mathop{\sum_{c_1=1}^{q_2} \cdots \sum_{c_t=1}^{q_2}}_{(c_1, \ldots, c_t, q_2)=1} \prod_{i=1}^s W_i(q_1,\mathbf{b}) W_i(q_2, \mathbf{c})
\]
\[
	= A(q_1)A(q_2).
\]
\end{proof}

Define the completed singular series 
\[
	\mathfrak{S} = \sum_{q=1}^\infty A(q).
\]
Since $A(q)$ is multiplicative, we have
\[
	\mathfrak{S} = \prod_p \left(1 + \sum_{\ell=1}^\infty A(p^\ell)\right).
\]

\begin{lemma}\label{singular_series_convergence} $\mathfrak{S}$ converges absolutely and there exists $\varepsilon>0$ such that $\mathfrak{S} - \mathfrak{S}(Q) \ll Q^{-\varepsilon}$.
\end{lemma}
\begin{proof}
\[
	A(p^\ell) = \mathop{\sum_{a_1=1}^{p^\ell} \cdots \sum_{a_t=1}^{p^\ell}}_{(a_1, \ldots, a_t, p)=1} \frac{1}{\phi(p^\ell)^s} \prod_{i=1}^s W_i(p^\ell, \mathbf{a})
\]
\[
	= \frac{p^{\ell s}}{\phi(p^\ell)^s} \mathop{\sum_{a_1=1}^{p^\ell} \cdots \sum_{a_t=1}^{p^\ell}}_{(a_1, \ldots, a_t, p)=1} \prod_{i=1}^s W_i(p^\ell, \mathbf{a}).
\]

This sum is now the one studied in the proof of Lemma 4.1 of~\cite{brandes_parsell}, so using their bound gives
\[
	A(p^\ell) \ll \left(1 - \frac{1}{p}\right)^{-s} (p^\ell)^{-1-\varepsilon} \ll (p^\ell)^{-1-\varepsilon}
\]
so
\[
	\mathfrak{S} = \prod_p \left(1 + \sum_{\ell=1}^\infty A(p^\ell)\right)
\]
converges.

\end{proof}


Define $M(q)$ to be the number of solutions to the simultaneous congruences
\[
	\sum_{i=1}^s u_{ij}m_i^{k_j} \equiv 0 \pmod q, \qquad (1 \leq j \leq t),
\]
with $(m_i, q)=1$ for all $i$.

\begin{lemma}\label{M_to_A} For any positive integer $q$,
\[
	M(q) = \frac{\phi(q)^s}{q^t} \sum_{d | q} A(d).
\]
\end{lemma}
\begin{proof}
By orthogonality,
\[
	M(q) = \frac{1}{q^t} \sum_{r_1=1}^q \cdots \sum_{r_t=1}^q \prod_{i=1}^s \sum_{\substack{m_i=1 \\ (m_i,q)=1}}^q e\left( \frac{\mathbf{r} \mathbf{u}_i m_i^{\mathbf{k}}}{q} \right)
\]
Let $d = \frac{q}{(r_1, \ldots, r_t, q)}$ and let $a_j = \frac{r_j}{(r_1, \ldots, r_t, q)}$.  Rearranging the sum according to the value of $d$ yields
\[
	M(q) = \frac{1}{q^t} \sum_{d | q} \mathop{\sum_{a_1=1}^q \cdots \sum_{a_t=1}^q}_{(a_1, \ldots, a_t, q)=1} \prod_{i=1}^s \frac{\phi(q)}{\phi(d)} \sum_{\substack{m_i = 1 \\ (m_i, d)=1)}}^d e\left(\frac{\mathbf{a} \mathbf{u}_i m_i^\mathbf{k}}{d}\right)
\]
\[
	= \frac{\phi(q)^s}{q^t} \sum_{d | q} A(d).
\]
\end{proof}

\begin{lemma}\label{hensel_lift} For positive integers $\beta, \gamma$ with $\beta > \gamma$,
\[
	M(p^\beta) \geq M(p^\gamma)p^{(\beta-\gamma)(s-2)}.
\]
\end{lemma}

This is a standard Hensel's lemma lifting.  Cf.~\cite{wooley_SAE_ii}, Lemma 6.7 and~\cite{brandes_parsell}, Lemma 4.2.

\begin{lemma}\label{singular_series_positivity} Given that the system (\ref{the_system}) has a solution modulo each prime $p$, $\mathfrak{S} > 0$.
\end{lemma}
\begin{proof}
By Lemma~\ref{M_to_A},
\[
	1 + \sum_{\ell = 1}^\infty A(p^\ell) = \lim_{\gamma \to \infty} \frac{p^{t\gamma}}{\phi(p^\gamma)^s} M(p^\gamma)
\]
\[
	\geq \lim_{\gamma \to \infty} p^{\gamma(t-s)}M(p^\gamma).
\]
By Lemma~\ref{hensel_lift}, this is 
\[
	\geq \lim_{\gamma \to \infty} p^{\gamma(t-s)}p^{(\gamma-1)(s-2)}M(p).
\]
Now since we assume $M(p) \ge 1$, this is 
\[
	\mathfrak{S} = 1 + \sum_{\ell = 1}^\infty A(p^\ell) \geq \lim_{\gamma \to \infty} p^{\gamma(t-s)}p^{(\gamma-1)(s-2)}
\]
\[
	\geq p^{\gamma t - s - 2\gamma + 2} > 0.
\]
Since
\[
	\mathfrak{S} = \prod_p \left(1 + \sum_{\ell = 1}^\infty A(p^\ell)\right),
\]
the lemma follows.
\end{proof}

\section{The Main Theorem}\label{sec:main_theorem}

We now have all of the necessary components for the proof of the main theorem.

\noindent\textit{Proof of Theorem~\ref{main_theorem}.}  By (\ref{R_breakdown}),
\[
	R(P) = \int_\mathcal{A} \prod_{i=1}^s f_i(\bm\alpha) d\bm\alpha.
\]
Since $\mathcal{A}$ is the disjoint union of $\mathfrak{M}(Q)$ and $\mathfrak{m}(Q)$, we may write
\[
	R(P) = \int_{\mathfrak{M}(Q)} \prod_{i=1}^s f_i(\bm\alpha) d\bm\alpha + \int_{\mathfrak{m}(Q)} \prod_{i=1}^s f_i(\bm\alpha) d\bm\alpha.
\]
Theorem~\ref{minor_arc_bound} bounds the integral over the minor arcs, yielding
\[
	R(P) = \int_{\mathfrak{M}(Q)} \prod_{i=1}^s f_i(\bm\alpha) d\bm\alpha + O(P^{s-K-\varepsilon}).
\]
Theorem~\ref{major_arc_breakdown} establishes a Hasse principle breakdown of the major arc integral, giving us
\[
	R(P) = C\mathfrak{S}(Q)J(Q) + O(P^{s-K-\varepsilon}).
\]
Lemma~\ref{singular_integral_bound} gives the asymptotic growth rate of the singular series $J(Q)$ and Lemmas~\ref{singular_series_convergence} and~\ref{singular_series_positivity} show that, under the assumptions of the theorem, the singular series $\mathfrak{S}(Q)$ converges to a positive constant.  We conclude that, for some constant $C>0$,
\[
	R(P) = CP^{s-K} + O(P^{s-K-\varepsilon}).
\]
\qed

\section{Acknowledgements}\label{sec:acknowledgements}

A weaker version of this result, requiring the assumption of a generalized Riemann Hypothesis, was proved in the author's doctoral dissertation, which was completed at Pennsylvania State University under the direction of Bob Vaughan.  The author thanks Bob for much guidance and instruction on both the dissertation and this work, particularly with respect to the ideas of sections~\ref{sec:exp_int} and~\ref{sec:minor_arcs}.

Much of this work was completed at the University of Waterloo under the guidance of Yu-Ru Liu and Wentang Kuo.  The author thanks Yu-Ru and Wentang for their support.

\end{document}